\documentclass[12pt]{article}
\usepackage[dvips]{epsfig}
\usepackage{amscd}
\usepackage{amssymb}
\usepackage{amsthm}
\usepackage{amsmath}
\usepackage{latexsym}

\theoremstyle{plain}
\newtheorem{thm}{Theorem}[section]
\theoremstyle{plain}
\newtheorem{lemma}[thm]{Lemma}

\theoremstyle{definition}

\newtheorem*{remark}{Remark}

{%
\setcounter{enumi}{0}

\begin{enumerate}}%
{\end{enumerate} }
{%
\setcounter{enumi}{0}

\begin{enumerate}}%
{\end{enumerate} }

\newcommand{\norm}[1]{\ensuremath{\left\|#1\right\|}}

\newcommand{\Om}{\ensuremath{\Omega}}

\newcommand{\eps}{\ensuremath{\varepsilon}}
\newcommand{\R}{\ensuremath{\mathbb{R}}}

\newcommand{\be} {\begin{equation}}
\newcommand{\ee} {\end{equation}}
\newcommand{\bea} {\begin{eqnarray}}
\newcommand{\eea} {\end{eqnarray}}
\newcommand{\Bea} {\begin{eqnarray*}}
\newcommand{\Eea} {\end{eqnarray*}}

\newcommand{\grad}{\ensuremath{\nabla}}
\newcommand{\De} {\Delta}

\newcommand{\pl} {\partial}

\catcode`\@=11

\numberwithin{equation}{section} \allowdisplaybreaks
\begin{document}
%
%

\title{On the solutions of a singular elliptic equation concentrating on two orthogonal spheres }

\author{ B. B. Manna\footnote{ B. B. Manna, TIFR CAM , Bangalore, email: bhakti@math.tifrbng.res.in} ,
P. N. Srikanth \footnote{P. N. Srikanth ,TIFR CAM , Bangalore, email: srikanth@math.tifrbng.res.in}}


\date{\today}
\maketitle

\begin{abstract} Let $A=\{x\in \R^{2m} : 0< a< |x| <b\}$ be an annulus. Consider the following singularly perturbed elliptic problem on $A$
 \begin{equation}
  \begin{array}{lll}
  -\eps^2{\De u} + |x|^{\eta}u = |x|^{\eta}u^p,  &\mbox{ \qquad in } A \notag\\
  u>0   &\mbox{ \qquad in }  A \\
  u = 0  &\mbox{   \qquad on } \partial A\\
  \end{array}
 \end{equation}
$1<p<2^*-1$. 
We shall prove the existence of a positive solution $u_\eps$ which concentrates on two different orthogonal spheres of dimension $(m-1)$ as
$\eps\to 0$. We achieve this by studying a reduced problem on an annular domain in $\R^{m+1}$ and analyzing the profile of a two point concentrating 
solution in this domain.
\end{abstract}

\section{Introduction}
Consider the following singularly perturbed elliptic equation with super linear nonlinearity on an annulus in $\R^{2m}$
 \begin{equation}
  \begin{array}{lll}
  -\eps^2{\De u} + |x|^{\eta}u = |x|^{\eta}u^p,  &\mbox{ \qquad in } A \label{E0.1}\\
  u>0   &\mbox{ \qquad in }  A \\
  u = 0  &\mbox{   \qquad on } \partial A\\
  \end{array}
 \end{equation}
$1<p<2^*-1$, $\eps$ is a singular perturbation parameter. $A=\{x\in \R^{2m} : 0< a< |x| <b\}$. $\eta=\frac{2m-2}{m-2}$. Let us take a suitable 
polar co-ordinate on the annulus $A$ as \be A=I\times( J\times S^{m-1} \times S^{m-1}).\notag\ee Where $I=(a,b)$, $J=[0,\pi/2)$ and $S^{m-1}$ 
has the standard polar co-ordinate representation. For $x\in A$ we can write 
$x\equiv x(r,\theta,\theta_1^1,\theta_2^1,\dots,\theta_{m-1}^1, \theta_1^2,\theta_2^2\dots,\theta_{m-1}^2)$, where
$r\in I$, $\theta\in J$ and $\theta_1^i\in [0,2\pi)$ for $i=1,2$ and $\theta_j^i\in[0,\pi)$ for $i=1,2$ and $j=2,\dots,m-1$. In this paper we 
shall prove the following result

\begin{thm}\label{Th0.1}
 The equation (\ref{E0.1}) has a solution $u_\eps$ which concentrates on two orthogonal spheres as $\eps \to 0$. The radial co-ordinate of the 
 spheres is $(a+b)/2$ and placed at the angle $\theta=0$ and $\theta=\pi/2$. 
\end{thm} 
This is the first result of its kind where one shows the concentration along two spheres which are orthogonal to each other.

The existence of least energy solution of such singularly perturbed elliptic problem concentrating at a point has been well studied in the papers
\cite{MR1115095},\cite{MR1219814},\cite{MR1639546},\cite{MR1342381}. In \cite{MR1342381} the authors have studied similar type of problem with
Dirichlet data and have shown that the least energy solution the concentrates at a single point as $\eps\to 0$ and the point of concentration 
converges to a point with the maximum distance from the boundary of the domain. Also in \cite{MR1115095} and \cite{MR1219814} the same problem 
has been studied for Neumann data. There authors have proved that for $\eps$ small enough the unique point of maxima goes to the point with
maximum mean curvature of the boundary.  
 
Different type of $S^n$ concentration phenomena has been studied by several authors. One of the pioneer works has been done by Ruf and srikanth 
\cite{MR2608946} where the authors have considered the similar type of problem in 4-dimension and have shown the solutions concentrate on a circle. 
Pacella and Srikanth in \cite{arXiv:1210.0782} have shown the $(m-1)$ dimensional spherical concentration for the similar type of problem.

Here in this work we shall first reduce the problem on an annulus in $\R^{m+1}$ first. There we shall look for least energy solution in an 
appropriate space and shall show the existence of solution which concentrates exactly at two points. Dancer and Yan in \cite{MR1671254}, have
shown the existence of multi peak solution for domains with holes.  See the work of Monica Clapp et al. in \cite{arXiv:1212.5137} for similar 
type of result. However our result, even it is specific, gives precise information of the concentrating solution and their profile.  

Also we can extend our result for any $\eta$ and we can prove the following result
\begin{thm}\label{Th0.2}
 The equation (\ref{E0.1}) has a solution $u_\eps$ which concentrates on two orthogonal spheres as $\eps \to 0$. The spheres are placed at the 
 angle $\theta=0$ and $\theta=\pi/2$ and radial co-ordinate of the spheres is 
 \begin{itemize}
  \item [(i)] $r=a$ if $\eta<\frac{2m-2}{m-2}$. 
  \item [(ii)] $r=(a+b)/2$ if $\eta=\frac{2m-2}{m-2}$.
  \item [(iii)] $r=b$ if $\eta>\frac{2m-2}{m-2}$.
 \end{itemize}

\end{thm}

We can express a point $x\in A$ in precise as $x=(x_1,x_2), x_i\in \R^m$ Let $\rho_1=r \ cos\theta$ and $\rho_2=r \ sin\theta$. 
Then we can express $x_1, x_2$ as $x_1\equiv x_1(\rho_1,\sigma_1)$ and $x_2\equiv x_2(\rho_2,\sigma_2)$, $\sigma_i\in S^{m-1}\subset \R^m$ 
Thus finally we can express $x$ precisely as \be x\equiv x(r,\sigma_1,\sigma_2,\theta)\label{0.1}\ee. The expression of the Laplacian in this 
co-ordinate given as
\be \Delta_{\R^{2m}}u=u_{rr}+\frac{2m-1}{r}u_r+\frac{m-1}{r^2}u_\theta\Big[\frac{2cos2\theta}{sin2\theta}\Big]+\frac{u_{\theta\theta}}{r^2}+\sum_{i=1}^2\frac{1}{r^2cos^2\theta}\Delta_{S^{m-1}}^{\sigma_i}u\label{0.2}\ee
where $\Delta_{S^{m-1}}^{\sigma_i}$ is the Laplace-Beltrami operator on $S^{m-1}$ in $\sigma^i$ variable. 
Define
\be X=\{u\in H_0^1(\overline{A}) : u \text{ is independent of } \theta_1^2,\theta_2^2\dots,\theta_{m-1}^2\}\label{0.3}\ee Then for $u\in X$
we have 
\be  \Delta_{\R^{2m}}u=u_{rr}+\frac{2m-1}{r}u_r+\frac{m-1}{r^2}u_\theta\Big[\frac{2cos2\theta}{sin2\theta}\Big]+\frac{u_{\theta\theta}}{r^2}+\frac{1}{r^2cos^2\theta}\Delta_{S^{m-1}}^{\sigma_1}u\label{0.4}\ee
Consider the riemannian submersion $\varphi : \overline{A}\to\overline{\Om}$ given by
\be \varphi (x(r,\sigma_1,\sigma_2,\theta))=y(s,\sigma_1,2\theta).\notag\ee Where $\Om$ is the annulus $I'\times S^m$, $I'=s(I)$ and 
$s(r)=\Big[\frac{m(m-2)^2}{(2m-3)(2m^2-4m+1)}\Big]^{1/2}r^{\frac{2m-3}{m-2}}$. For $u\in X$ we can define 
$v(s,\sigma_1,2\theta):=u(r,\sigma_1,\sigma_2,\theta)$. Then for $\eta=\frac{2m-2}{m-2}$, we can easily check that , $u$ satisfies (\ref{E0.1}) iff $v$ satisfies
 \begin{equation}
  \begin{array}{lll}
  -\eps^2{\De v} + v = v^p,  &\mbox{ \qquad in } \Om \label{E0.2}\\
  v>0   &\mbox{ \qquad in }  \Om \\
  v = 0  &\mbox{   \qquad on } \partial \Om\\
  \end{array}
 \end{equation}
 Our aim is to prove the existence of two-peak solutions of (\ref{E0.2}) and the location of the peaks as $\eps\to 0$. From this the theorem\ref{Th0.1}
 shall follow easily.

\section{Existence of two peak solutions}

Let us define $H_\sharp(\Om)\subset H_0^1(\Om)$ as 
\be H_\sharp(\Om)=\{u\in H_0^1(\Om) : u(x_1,\dots,x_m,x_{m+1})\equiv u(\sqrt{x_1^2+\dots+x_m^2}, |x_{m+1}|)\}.\notag\ee 

Note that any solution in $H_\sharp$ shall have at least two local maxima (antipodal points). We shall show that there exists a solution with 
exactly two local maximas for $\eps \ll 1$.

\begin{lemma}\label{L1.1}
 $H_\sharp(\Om)$ is a closed subspace of $H_0^1(\Om)$ and $H_{0,rad}^1(\Om)\subset H_\sharp(\Om).$
\end{lemma}
The energy functional $J_\eps(u)$ is defined as
\be J_\eps(u)=\int_{\Om}(\frac{\eps^2}{2}|\grad u|^2 + \frac{u^2}{2})dx-\frac{1}{p+1}\int_{\Om}u_+^{p+1}dx\label{enf}\ee

\begin{lemma}\label{L1.2}
 The Morse index of any radial solution of (\ref{E0.2}) is $\ge 2$.
\end{lemma}
\begin{proof}
 Let $u\in H_{0,rad}^1(\Om)$ satisfies (\ref{E0.2}). Then
 \be D^2 J_\eps(u)[u,u]=\int_{\Om}(\eps^2|\grad u|^2 +u^2)dx-p\int_{\Om}u^{p+1}dx<0\notag\ee as $p>1$ and $u \neq 0$. 
 
 Now let $v=u \ cos(2\theta)$ where the coordinate system of $\Om$ is taken as standard polar co-ordinate. Note that $u$ being radial we have
 $v\equiv v(r,\theta)$. Then $\Delta v$ in standard polar co-ordinate takes the form
 \be \Delta v=v_{rr}+\frac{m}{r}v_r+\frac{1}{r^2}v_{\theta\theta}+\frac{m-1}{r^2}\frac{cos\theta}{sin\theta}v_\theta.\notag\ee
 So we have
 \begin{align}
  \Delta v&=(u_{rr}+\frac{m}{r}u_r)cos2\theta-\frac{4}{r^2}u \ cos2\theta-\frac{2(m-1)}{r^2}\frac{cos\theta}{sin\theta}sin2\theta \ u\notag\\
  -\eps^2\Delta v &= (-\eps^2\Delta u)cos2\theta+\frac{4\eps^2}{r^2}u \ cos2\theta+\frac{4(m-1)\eps^2}{r^2}cos^2\theta \ u\notag\\
  &=(-u+u^p)cos2\theta+\frac{4\eps^2}{r^2}u \ cos2\theta+\frac{4(m-1)\eps^2}{r^2}cos^2\theta \ u\notag
 \end{align}

 Hence
 \begin{align}
  &\langle-\eps^2\Delta v+(1-pu^{p-1}v,v\rangle\notag\\
  =&(1-p)\int_{\Om}u^{p+1}cos^22\theta dx+\int_{\Om}\frac{4\eps^2}{r^2}u^2 \ cos^22\theta dx+\int_{\Om}\frac{4(m-1)\eps^2}{r^2}cos^2\theta cos2\theta\ u^2 dx\notag\\
  =& C\int_a^b\int_0^\pi\Big[(1-p)r^{m}u^{p+1}cos^22\theta+4\eps^2r^{m-2}u^2 \ cos^22\theta+4(m-1)\eps^2r^{m-2}u^2cos^2\theta cos2\theta\Big]drd\theta\notag
 \end{align}
 
 Integrating over $\theta$ and using $a>0$ as the lower bound of $r$ we get, for some positive constant(generic) $C$
 \be \langle-\eps^2\Delta v+(1-pu^{p-1}v,v\rangle\le C\int_a^b \Big[(1-p)u^{p+1}+C\eps^2u^2\Big]r^mdr\notag\ee
 $u$ being a radial solution of (\ref{E0.2}) we have 
 \be \int_a^b[u^{p+1}-u^2]r^mdr=\int_a^b|u_r|^2r^mdr>0\notag\ee
 Hence the result follows by taking $\eps$ small enough such that $(1-p)+C\eps^2<0$. So $u$ and $v$ being orthogonal we see that the 
 Morse index of any radial solution $u$ is greater or equal to 2.
 \end{proof}

\begin{remark}
Note that $H_\sharp$ is closed in $H_0^1(\Om)$ (hence compactly embedded in $L^2(\Om)$). It can be easily shown that $J_\eps$ satisfies $(PS)_c$
condition for any critical value $c$. Also it has M-P geometry near the origin. Next we shall construct a test function, from which we shall get
a M-P critical level $c_\eps$. Also using the test function we shall get an upper bound of $c_\eps$ which shall help us analyzing the behavior
of a sequence of solutions as $\eps\to0$. 
\end{remark}

Consider the following equation
\begin{equation}
 \left\{\begin{aligned}
  &-\Delta U+U=U^p &&\text{ in } \R^{m+1}\label{LE}\\
  & U>0 &&\text{ in } \R^{m+1}\\
  &\lim_{|x|\to\infty}U(x)=0 && U(0)=\max_{x\in \R^{m+1}}U(x) 
 \end{aligned}
 \right.
\end{equation}

The energy functional $J(U)$ of (\ref{LE}) given by 
\be J(U)=\frac{1}{2}\int_{\R^{m+1}}(|\grad U|^2+U^2)dx+\frac{1}{p+1}\int_{\R^{m+1}}U^{p+1}dx, \ U\in H^1(\R^{m+1})\label{LE1}\ee
It is well known that, there exists a least positive critical value $c^*$ characterize by 
\be c^*=\inf_{v\neq 0}\max_{t>0}J(v)\label{ELe}\ee. And there is a least energy solution $U$ of (\ref{LE}) such that $U(x)=U(|x|)$ and 
\be |D^\alpha U(x)|\le Cexp(-\delta|x|)\label{DcyLe}\ee for some $c,\delta>0$ and any $|\alpha|\le 2$. Furthermore $U$ satisfies the Pohozaev 
identity \be \frac{m-1}{2}\int_{\R^{m+1}}|\grad U|^2dx+\frac{m+1}{2}\int_{\R^{m+1}}U^2dx+\frac{m+1}{p+1}\int_{\R^{m+1}}U^{p+1}(y)dx=0\label{poho}\ee

Define \be Z_{\eps,t}^\gamma = \Phi_\gamma\Big(\frac{x-Q}{t}\Big)U\Big(\frac{x-Q}{\eps t}\Big)+\Phi_\gamma\Big(\frac{x+Q}{t}\Big)U\Big(\frac{x+Q}{\eps t}\Big)\label{TF}\ee
where $Q=(0,\dots,0,\frac{a+b}{2})\in \Om$ ($I'=(a,b)$ say) and $\phi_\gamma$ is a non-negative smooth radial function supported in $B(0,2\gamma)$ and $\grad\phi_\gamma <2/\gamma$
and 
\be
\phi_\gamma(r)=\left\{\begin{aligned}
  &1 && \text{ for }0\le r\le\gamma\notag\\
  &0 &&\text{ otherwise }.
 \end{aligned}
 \right.
\ee
Also $\gamma$ is choose so that $B(0,2t\gamma)\subset B(0,\frac{b-a}{2})$. We can easily show that $Z_{\eps,t}^\gamma\in H_\sharp(\Om)$

Let us calculate the energy at $Z_{\eps,t}^\gamma$. First note that the supports of $\Phi_\gamma\Big(\frac{x-Q}{t}\Big)U\Big(\frac{x-Q}{\eps t}\Big)$
and $\Phi_\gamma\Big(\frac{x+Q}{t}\Big)U\Big(\frac{x+Q}{\eps t}\Big)$ are disjoint.
\begin{align}
 J_\eps(Z_{\eps,t}^\gamma)&=\int_{\Om}\Big(\frac{\eps^2}{2}|\grad Z_{\eps,t}^\gamma|^2 + \frac{|Z_{\eps,t}^\gamma|^2}{2}\Big)dx-\frac{1}{p+1}\int_\Om |Z_{\eps,t}^\gamma|^{p+1}dx\notag\\
\end{align}

To get the mountain-pass solution we shall calculate the above energy explicitly. Denote $\phi_\gamma^{\pm}\equiv \phi_\gamma(\frac{1}{t}(x\pm Q))$ 
and $U^{\pm}\equiv \phi_\gamma(\frac{1}{\eps t}(x\pm Q))$ Then
\begin{align}
 &\int_{\Om}\frac{\eps^2}{2}|\grad Z_{\eps,t}^\gamma|^2dx\notag\\
 =& \int_{\Om}\frac{\eps^2}{2}\Big |\Big(\frac{1}{t}\grad \phi_\gamma^-U^-+\frac{1}{\eps t}\phi_\gamma^-\grad U^-\Big)+\Big(\frac{1}{t}\grad \phi_\gamma^+U^++\frac{1}{\eps t}\phi_\gamma^+\grad U^+\Big)\Big|^2dx\notag\\
 =& \int_{\Om}\frac{\eps^2}{2}\Big |\frac{1}{t}\grad \phi_\gamma^-U^-+\frac{1}{\eps t}\phi_\gamma^-\grad U^-\Big|^2+\int_{\Om}\frac{\eps^2}{2}\Big|\frac{1}{t}\grad \phi_\gamma^+U^++\frac{1}{\eps t}\phi_\gamma^+\grad U^+\Big|^2dx\notag\\
 =& I_1+I_2 \text{ (say) }\notag
\end{align}

Make the change of variable $y=\frac{1}{\eps t}(x-Q)$ in $I_1$ and $y=\frac{1}{\eps t}(x+Q)$ in $I_2$ and using the decay estimate of the solution
$U$ of (\ref{E0.2}) we get
\begin{align}
 &\int_{\Om}\frac{\eps^2}{2}|\grad Z_{\eps,t}^\gamma|^2dx\notag\\
 =& \eps^3\int_{B(0,2\gamma/\eps)}\Big(t\phi_\gamma^2(\eps y)|\grad U|^2+2t\eps \phi_\gamma(\eps y) \grad \phi_\gamma(\eps y)U(y)\grad U(y)+t\eps^2|\grad \phi_\gamma(\eps y)|^2U^2(y)\Big)dy\notag\\
 =& \eps^3\Big(t\int_{B(0,2\gamma/\eps)}\phi_\gamma^2(\eps y)|\grad U|^2dy+2t\eps\int_{B(0,2\gamma/\eps)}2t\eps \phi_\gamma(\eps y) \grad \phi_\gamma(\eps y)U(y)\grad U(y)dy+O(\eps^2)\Big)\notag\\
 =& \eps^3 \Big(t\int_{\R^{m+1}}|\grad U|^2dx+ O(\eps)\Big) \notag
\end{align}

The Second term is 
\begin{align}
 \int_\Om\frac{1}{2}| Z_{\eps,t}^\gamma|^2dx & = \int_{A}\frac{1}{2}\Big |\phi_\gamma^-U^-+\phi_\gamma^+ U^+\Big|^2dx\notag\\
 &=t^3\eps^3\int_{B(0,2\gamma/\eps)}\phi_\gamma^2(\eps y)U^2(y)dy\notag\\
 &=\eps^3\Big[t^3\int_{\R^{m+1}}U^2(y)dy+O(\eps)\Big]\notag
\end{align}
Similarly
\begin{align}
 \frac{1}{p+1}\int_\Om| Z_{\eps,t}^\gamma|^{p+1}= \frac{2}{p+1}\eps^3\Big[t^3\int_{\R^{m+1}}U^{p+1}(y)dy+O(\eps)\Big]\notag
\end{align}

Combining all those terms we have 
\begin{align}
 \eps^{-3}J_\eps(Z_{\eps,t}^\gamma)=2\Big[ \frac{t}{2}\int_{\R^{m+1}}|\grad U|^2dx+\frac{t^3}{2}\int_{\R^{m+1}}U^2dx+\frac{t^3}{p+1}\int_{\R^{m+1}}U^{p+1}(y)dx+O(\eps)\Big]\label{EE}
\end{align}

Then from (\ref{EE}) we have 
\be \eps^{-3}J_\eps(Z_{\eps,t}^\gamma)= -(\frac{t^3}{3}-t)\int_{\R^{m+1}}|\grad U|^2dx+O(\eps).\label{MP1}\ee
We choose $t_0$ such that \be -(\frac{t_0^3}{3}-t_0)\int_{\R^{m+1}}|\grad U|^2dx+O(\eps)<-1.\notag\ee
Now choose $\gamma$ such that $B(0,2t_0\gamma)\subset B(0,\frac{b-a}{2})$. The there exists $\eps_0$ such that 
\be J_\eps(Z_{\eps,t_0}^\gamma)<0 \text{ for all } \eps<\eps_0\label{MP2}\ee. We define 
\be c_\eps=\inf_{\beta\in\mathcal{P}}\max_{t\in[0,1]}J_\eps(\beta(t)).\label{MP3}\ee Where 
$\mathcal{P} = \{\beta\in C([0,1], H_\sharp): \beta(0)=0, \beta(1)=Z_{\eps,t_0}^\gamma\}$. From MP-Lemma we have $c_\eps$ as a critical value of 
$J_\eps$. 
\begin{remark}
It can be easily shown that $J_\eps$ satisfies all the hypothesis of [\cite{MR1251958}, Thm 10.2](also look at [\cite{MR1886088},p.222 and Thm 5.1] 
and [\cite{MR1303222}) ,p.1598]). Hence at the level $c_\eps$ there is a solution $u_\eps$ of (\ref{E0.2}) with Morse Index less or equal to one.
then we can readily see from lemma\ref{L1.2} that there is a non radial M-P solution corresponding to the critical value $c_\eps$. Also note that
the solution is not the least energy solution in the space $H_0^1(\Om)$.
\end{remark}

\section{profile of the solution}
In this section we follow the line of proof in J.Beyon and J.Park, \cite{MR2180862}. 
Let $\beta(t)=Z_{\eps,tt_0}^\gamma$. Then it follows that $\lim_{t\to 0}\beta(t)=0$ and $\beta(1)=Z_{\eps,t_0}^\gamma$. Note that 
$Z_{\eps,tt_0}^\gamma\in H_\sharp$ for all $t\in [0,1]$. More over from (\ref{EE}) we have
\be J_\eps(Z_{\eps,t}^\gamma)=2\eps^{3}\Big[ \frac{tt_0}{2}\int_{\R^{m+1}}|\grad U|^2dx+\frac{(tt_0)^3}{2}\int_{\R^{m+1}}U^2dx+\frac{(tt_0)^3}{p+1}\int_{\R^{m+1}}U^{p+1}(y)dx+O(\eps)\Big]\notag\ee
Hence we have 
\be \overline{\lim_{\eps\to 0}}\eps^{-3}c_\eps\le 2\max_{t\in(0,t_0)}\Big[ \frac{t}{2}\int_{\R^{m+1}}|\grad U|^2dx+\frac{t^3}{2}\int_{\R^{m+1}}U^2dx+\frac{t^3}{p+1}\int_{\R^{m+1}}U^{p+1}(y)dx\Big]\notag\ee
We can easily show that the maximum occurs at $t=1$ and hence we have 
\be \overline{\lim_{\eps\to 0}}\eps^{-3}c_\eps\le 2J(U).\label{3.1}\ee 
Now by Sobolev embedding and the bootstrap argument we can find $\{\norm{u_\eps}_\infty\}$ is bounded and by maximum principle we see that it is
bounded away from $0$. Also from Pohozaev's identity (\ref{poho}) we can have $\{\eps^{-3}\norm{u_\eps}^2_\eps\}$ is bounded where 
\be \norm{u_\eps}^2_\eps=\int_\Om(\eps^2|\grad u_\eps|^2+u^2)dx\label{nor}\ee
Let $q_\eps\in A$ such that $\lim\inf_{\eps\to 0}u_\eps(q_\eps)>0$. Consider \be \omega_\eps=u_\eps(\eps(x-q_\eps))\label{ScFn}\ee

\begin{lemma}
 $\frac{1}{\eps}dist(q_\eps,\pl A)\to\infty$ as $\eps \to 0$.
\end{lemma}
\begin{proof} Similar proof as given in proposition 4 of \cite{MR2608946}
\end{proof}
Then it can be easily shown that up to a subsequence $\omega_\eps$ converges (locally in $C^2$) to a finite energy solution $W$ of (\ref{LE}). Note
that $W$ is non-negative and hence by uniqueness $W=U$.

Also from (\ref{3.1}) we can say that, there can be at most finitely many $q_\eps^1,\cdots,q_\eps^k\in A$ satisfying
\begin{itemize}
 \item [(i)] $\lim_{\eps\to 0}\frac{q_\eps^i-q_\eps^j}{\eps}=\infty$ for $i\neq j$ and
 \item [(ii)] $\lim\inf_{\eps\to 0}u_\eps(q_\eps^i)>0$ for $j=1,2,\cdots, k.$
\end{itemize}
Then from comparison principle we get $c$ and $C>0$ such that \be u_\eps(x)+|\grad u_\eps(x)|\le C exp(-\frac{c}{\eps} dist(x,q_\eps^1,\cdots,q_\eps^k))\label{TA1}\ee
And this implies \be \liminf_{\eps\to 0}\eps^{-3}J_\eps(u_\eps)=kJ(U)\label{TA2}\ee
Then from (\ref{3.1}) we have $k\le2$. Note that as $u_\eps\in H_\sharp$, then for any point of maxima $q_\eps$ of $u_\eps$, the antipodal point
$-q_\eps$ is also a max of $u_\eps$ and $q_\eps$ must lie on $x_{m+1}-axis.$

\section{Location of the spike layers and the proof if Theorem\ref{Th0.1}}
We shall determine the location of the peaks by estimating the upper and lower bounds of the energy. Let $q_\eps$ be one of the maximum points of
the solution $u_\eps$. And it has been proved that $\frac{1}{\eps}dist(q_\eps,\pl A)\to\infty$ as $\eps \to 0$. Let $q_\eps\to\overline{q}$ as
$\eps\to 0$. Then another concentration point is $-\overline{q}$.

Consider the following problem in a ball 
$B_\rho=\{x\in\R^{m+1} : |x|\le \rho\}$, for some $\rho>0$.
  \begin{equation}
  \begin{array}{lll}
  -\Delta u + u = u^p,  &\mbox{ \qquad in } B_\rho \label{CpE}\\
  u>0   &\mbox{ \qquad in }  B_\rho \\
  u = 0  &\mbox{   \qquad on } \pl B_\rho\\
  \end{array}
 \end{equation}
This functional has a least positive critical value, denoted by $c_\rho$ , which can be characterized similarly to (\ref{MP3}). Using Schwarz’s 
symmetrization, we find at least one radially symmetric least energy solution of (\ref{ELe}).
\begin{lemma}
 \be c_\rho=c^*+exp[-2\rho(1+o(1))],\notag\ee where $c^*$ is given in (\ref{ELe})
\end{lemma}
\begin{proof}
 The proof can be found in lemma 2.1 in \cite{MR1736974}.
\end{proof}

We shall now use the previous lemma to give the upper and lower bound for $c_\eps$.

Let $U_\rho$ be a corresponding radial least energy solution to (\ref{CpE}). Define 
\be U_\rho^\eps=U_\rho(\eps(x-\overline{q}))+U_\rho(\eps(x+\overline{q}))\notag\ee Then we see that for any $\rho>0$ there is an $\eps_\rho$ such that 
$U_\rho^\eps\in H_\sharp$ for all $\eps<\eps_\rho$ and support of $U_\rho(\eps(x-\overline{q}))$ and $U_\rho(\eps(x+\overline{q}))$ are disjoint. 
Estimating as before we can have the energy at $U_\rho^\eps$ as 
\be J_\eps(U_\rho^\eps)=2\eps^3\int_{B_\rho(0)}\Big(\frac{1}{2}(|\grad U_\rho|^2+|U_\rho|^2)-\frac{1}{p+1}|U_\rho|^{p+1}\Big)=2\eps^3c_\rho\label{4.1}\ee
Hence we have \be c_\eps\le 2\eps^3c_\rho=2(c^*+exp[-2\rho(1+o(1))]),\label{4.2}\ee We can take $\rho=\frac{1}{\eps}dist(\overline{q},\pl \Om)$.
Then finally we have \be c_\eps\le 2\eps^3c_\rho=2\eps^3\Big(c^*+exp\Big[-\frac{2}{\eps}dist(\overline{q},\pl\Om)(1+o(1))\Big]\Big),\label{4.3}\ee

Next we shall estimate $c_\eps$ from below. Let \be d_\eps=dist(q_\eps,\pl \Om)=dist(-q_\eps,\pl \Om)\to d_0=dist(\bar{q},\pl \Om) \text{ as } \eps\to 0\notag\ee 
Given $\delta$ choose a number $d_0'>0$ such that \be vol(B(\bar{q},d_0'))=vol(\Om\cap B(\bar{q},d_0+\delta)).\notag\ee. Now take $\delta'>0$ 
slightly smaller than $\delta$ with $d_0'<d_0+\delta'$.

Let us consider a $C^\infty$ cutoff function $\eta_\eps(s)$ such that $\eta_\eps(s)=1$ for $0\le s\le d_\eps+\delta'$ and $\eta_\eps(s)=0$ for 
$s\ge d_\eps+\delta$, with $0\le \eta_\eps\le 1$ and with uniformly bounded derivative. Let us sate 
$\tilde{u}_\eps=u_\eps\eta_\eps(x-\bar{q})+u_\eps\eta_\eps(x+\bar{q})$. Note that the support of $u_\eps\eta_\eps(x-\bar{q})$ and 
$u_\eps\eta_\eps(x+\bar{q})$ are disjoint and we find that
\be c_\eps\ge J_\eps(tu_\eps)\ge J_\eps(t\tilde{u}_\eps)-2exp\Big[-\frac{c}{\eps}(d_\eps+\delta')\Big]\label{4.4}\ee for $\eps$ sufficiently small.
Here we have used the estimates in (\ref{TA1}). Now note that \be J_\eps(t\tilde{u}_\eps)= J_\eps(tu_\eps\eta_\eps(x-\bar{q}))+J_\eps(tu_\eps\eta_\eps(x+\bar{q}))\notag\ee 
So using the same analysis as in \cite{MR1736974} we can get $dist(q_\eps,\pl \Om)\to \max_{x\in \Om}dist(x,\pl \Om)$. That is $q_\eps$ concentrates at
$(0,0,\frac{a+b}{2})$.

\textbf{\textit{Proof of the Theorem\ref{Th0.1}}} : 
\begin{proof}
 Let $v_\eps$ be a non radial M-P solution of \ref{E0.2} at the critical level $c_\eps$ with
Morse Index less or equal to 1(actually it is 1 as $u_\eps$ being positive radial solution). Now from the above discussion we see that,the sequence
$\{v_\eps\}_\eps$ concentrates exactly at two points in $\Om$. Also we have got the co-ordinate of the those concentrating points as 
$(\frac{a+b}{2},0,0,\dots,0)$ and $(\frac{a+b}{2},\pi,0,\dots,0)$. Now corresponding to $v_\eps$ there is a solution $u_\eps$ of equation
(\ref{E0.1}). And we get that $\{u_\eps\}_\eps$ concentrates exactly at two $m-1$ dimensional spheres in $A$, placed at the angle $\theta=0$ and 
$\theta=\pi/2$ with the radial co-ordinate $(a+b)/2$ .
\end{proof}

\begin{remark}
 In this note we are not detailing the proof of Theorem\ref{Th0.2} since the aim of this note is just to present the basic idea involved. The 
 details of the proof can be found in the PHd. thesis of B.B.Manna which is under preparation.
\end{remark}

\end{document}